\documentclass[12pt]{article}

\usepackage{geometry} 
\usepackage{amsthm,amssymb}
\usepackage{amsmath,amsfonts,latexsym}
\usepackage{latexsym}
\usepackage{float}
\usepackage{xcolor}
\usepackage{url}
\usepackage{comment}

\newtheorem{defi}{Definition}
\newtheorem{prop}{Proposition}
\newtheorem{tm}{Theorem}
\newtheorem{lema}{Lemma}
\newtheorem{kor}{Corollary}

\theoremstyle{remark}
\newtheorem{remark}{Remark}
\newtheorem{ex}[remark]{Example}

\title{Constructions of self-orthogonal and LCD subspace codes}
\author{Dean Crnkovi\'c \thanks{Faculty of Mathematics, University of Rijeka, Radmile Matej\v ci\'c 2, 51000 Rijeka, Croatia. \texttt{deanc@math.uniri.hr}} \and
Keita Ishizuka \thanks{Information Technology R\&D Center, Mitsubishi Electric Corporation, Kanagawa, Japan. \texttt{ishizuka.keita@ce.mitsubishielectric.co.jp}} \and
 Hadi Kharaghani\thanks{Department of Mathematics and Computer Science, University of Lethbridge,
Lethbridge, Alberta, T1K 3M4, Canada. \texttt{kharaghani@uleth.ca}} \and 
Sho Suda \thanks{Department of Mathematics, National Defense Academy of Japan. 2-10-20 Hashirimizu, Yokosuka, Kanagawa, 239-8686, Japan. \texttt{ssuda@nda.ac.jp}} \and
Andrea \v Svob \thanks{Corresponding author. Faculty of Mathematics, University of Rijeka, Radmile Matej\v ci\'c 2, 51000 Rijeka, Croatia. \texttt{asvob@math.uniri.hr}}
}
\date{\today}

\begin{document}
\maketitle

\begin{abstract}
Recently, the notions of self-orthogonal subspace codes and LCD subspace codes were introduced, and LCD subspace codes obtained from mutually unbiased weighing matrices were studied. In this paper, we provide a method of constructing self-orthogonal and LCD subspace codes from a set of matrices under certain conditions. In particular, we give constructions of self-orthogonal and LCD subspace codes from mutually quasi-unbiased weighing matrices, linked systems of symmetric designs, and linked systems of symmetric group divisible designs, Deza graphs and their equitable partitions.
\end{abstract}




\section{Introduction}\label{intro_new}

In 2000, Ahlswede, Cai, Li and Yeung (see \cite{network}) introduced network coding.
A network is a directed multigraph consists of vertices of different types. The source vertices transmit messages to the sink vertices through a channel of inner vertices. 
The idea of network coding is to allow mixing of data at these intermediate network vertices. 
A receiver observes the data packets to deduce the messages that were originally intended for the sinks. 
A special case arises when the packets are interpreted as vectors of symbols from a finite field, and the mixing functions are linear transformations.
This is the case of linear network coding. 

In 2006, Ho, M\'{e}dard, K\"{o}tter, Karger, Effros, Shi and Leong demonstrated that the so-called multicast capacity of a network is achieved with high probability in a sufficiently large field by a random choice of local mixing functions (see \cite{random}). In that way, random  network coding was introduced. Since the linear transformation by the channel of the
transmitted vectors is not known in advance by the transmitter or any of the receivers, such a model is sometimes referred to as a non-coherent transmission model. To achieve information transmission in random linear network coding, one might seek a property of the transmitted vectors that is preserved by the operation of the channel. In 2008, in their seminal paper K\"{o}tter and Kschischang
considered information transmission not via the transmitted vectors, but rather by the vector space that they span (see \cite{network-coding}). These types of codes are called subspace codes.
K\"{o}tter and Kschischang proved that subspace codes are efficient for transmission in networks, and their applications in error correction for random network coding attract a wide attention in recent research. See for example \cite{daniele, leo-constant, zhang}.

Self-orthogonal subspace codes were introduced in \cite{dsa-as}, with a construction by association schemes.  
Further, LCD subspace codes were introduced in \cite{dc-as-lcd}, and it was shown that under certain conditions equitable partitions of association schemes yield LCD subspace codes. Constructions from mutually unbiased weighing matrices were also given in \cite{dc-as-lcd}. In this paper, we give constructions of self-orthogonal and LCD subspace codes from mutually quasi-unbiased weighing matrices, linked systems of symmetric designs, linked systems of symmetric group divisible designs, and the equitable partitions of these designs. 

The paper is outlined as follows. In the next section, we provide the relevant background information. In Section \ref{SO-LCD-subspace}, we give constructions of self-orthogonal and LCD subspace codes from a set of integer matrices satisfying a certain assumption given in Theorem~\ref{lcd-matrices}. Examples satisfying the condition of the theorem include mutually quasi-unbiased weighing matrices, linked systems of symmetric designs and linked systems of symmetric group divisible designs, Deza graphs, and Deza digraphs of type II. 
Deza digraphs of type II were introduced and constructed in \cite{deza}. In Appendix A, we modify the construction to obtain Deza graphs, which is of independent interest.   
Further, we apply the construction to association schemes and give some examples. 


\section{Preliminaries} \label{intro}

We assume that the reader is familiar with the basic facts of coding theory (see \cite{FEC}). For recent results on subspace codes we refer the reader to \cite{network-book, heinlein-dcc}. 
Throughout, $\mathbb{F}_{q}$ denotes the finite field of order $q$ for some prime power $q$.

A $q$-ary \emph{linear code} $C$ of length $n$ and dimension $k$ is a $k$-dimensional subspace of a vector space $\mathbb{F}_{q}^n$ over $\mathbb{F}_{q}$ where $q$ is a prime power. A $q$-ary linear code of length $n$, dimension $k$, and minimum distance $d$ is called an $[n, k, d]_q$ code.
We say a $k \times n$ matrix comprised of rows that span $C$ generates $C$, or is a \emph{generator matrix} of $C$. 
The \emph{dual} code $C^\perp$ of a code $C$ is defined to be $C^{\perp} = \{ v \in \mathbb{F}_{q}^n \mid \langle v,c\rangle=0 {\rm \ for\ all\ } c \in C \}$, where $\langle x,y\rangle=\sum_{i=1}^n x_iy_i$ for $x=(x_i)_{i=1}^n,y=(y_i)_{i=1}^n\in\mathbb{F}_q^n$. 
A code $C$ is \emph{self-orthogonal} if $C \subseteq C^\perp$ and \emph{self-dual} if the equality is attained.   
Let $G$ be a generator matrix for a $q$-ary linear code $C$. The code $C$ is self-orthogonal if and only if 
$G G^{\top}=O$  
over the finite field $\mathbb{F}_{q}$, where $O$ denotes the zero matrix. 

Linear codes with complementary duals, LCD codes for short, are linear codes whose intersection with their duals are trivial.
LCD codes were introduced in 1992 by Massey (see \cite{massey}), and have been widely applied in information protection, electronics and cryptography.
Let $G$ be a generator matrix for a linear code $C$ over the field $\mathbb{F}_{q}$. The code $C$ is an LCD code if and only if $\det(GG^{\top}) \neq 0$ in $\mathbb{F}_{q}$.

In the classical coding theory, the main problem is to determine the largest minimum distance of codes for a given dimension and a length, and if possible, classify the optimal codes, i.e., the codes whose 
minimum distances achieve the theoretical upper bound on the minimum distance of $[n, k]_q$ codes.

In the subspace coding theory, the main problem is the same. In what follows, we give the basic definition from the theory of subspace codes.

A {\it subspace code} $\mathcal{C}$ is a nonempty set of subspaces of $\mathbb F^n_q$. Two metrics that are frequently used for subspace codes are subspace distance and injection distance (see \cite{network-book}).
The {\it subspace distance} is given by
\begin{equation} \label{submet}
d_S(U,W)=\dim(U+W)-\dim(U\cap W),
\end{equation}
where $U, W \in \mathcal{C}$. The {\it injection distance}, introduced in 2008 by K\"{o}tter, Kschischang and Silva (see \cite{rank-metric}), is given as follows
\begin{equation} \label{injmet}
d_I(U,W)= \max \{\dim(U), \dim(V) \} - \dim(U\cap W).
\end{equation}

In this paper, we will use the subspace distance. Hence, the {\it minimum distance} of a subspace code $\mathcal{C}$ is given by 
$$d=\min \{d_{S}(U,W) \mid  U,W \in \mathcal{C}, U \neq W \}.$$

A subspace code $\mathcal{C}$ composed of subspaces of the vector space $\mathbb F^n_q$ is called an $(n,\#\mathcal{C},d;K)_q$ subspace code if the minimum distance of $\mathcal{C}$ is $d$ and the dimensions of the codewords of $\mathcal{C}$ are contained in the set 
$K\subseteq \{0,1,\dots, n\}$. 
In the case $K=\{k\}$, a subspace code $\mathcal{C}$ is called a {\it constant dimension code} with the parameters $(n,\#\mathcal{C},d;k)_q$, otherwise, \textit{i.e.} if there are two codewords that do not have the same dimension, $\mathcal{C}$ is called a {\it mixed dimension code}. Such a subspace code is sometimes denoted by $(n,\#\mathcal{C},d)_q$. Constant dimension codes, being the $q$-analogs of the constant-weight codes, are the most studied subspace codes. 

Recently, in \cite{dsa-as}, the authors introduced self-orthogonal subspace codes.

\begin{defi}
Let ${\cal P}_q(n)$ be the set of all subspaces of $\mathbb F_q^n$. The {\it dual} code of a subspace code $\mathcal{C} \subseteq  {\cal P}_q(n)$ is the set $\mathcal{C}^{\perp}$ of all vector spaces in 
${\cal P}_q(n)$ that are orthogonal to each vector space in $\mathcal{C}$.
If $\mathcal{C} \subseteq \mathcal{C}^\perp$, then $\mathcal{C}$ is called a {\it self-orthogonal} subspace code. 
\end{defi}

The notion of LCD subspace codes was introduced in 2023 by the first and the fifth authors of the paper \cite{dc-as-lcd}.

\begin{defi}
Let $\mathcal{C}$ be a subspace code. If $C_i \cap C_j^\perp=\{0\}$ for all $C_i,C_j \in \mathcal{C}$, then $\mathcal{C}$ is called an {\it LCD subspace code}.
\end{defi}

Each codeword of an LCD subspace code $\mathcal{C}$ is an LCD code, since for every $C_i \in \mathcal{C}$ it holds that $C_i \cap C_i^\perp=\{0\}$.

Let $\mathcal{C}= \{ C_1, C_2, \ldots, C_{m} \}$ be an $(n,m,d;K)_q$ subspace code, and $G_i$ be a generator matrix of $C_i$, $1 \le i \le m$. The subspace code $\mathcal{C}$ is self-orthogonal if and only if
{$G_i G_j^{\top}=0$ for all $1 \le i, j \le m$. If 
$G_i G_j^{\top}$ is nonsingular for all $1 \le i, j \le m$, then $\mathcal{C}$ is an LCD subspace code.

In \cite{dsa-as}, the authors gave examples of self-orthogonal subspace codes, and examples of LCD codes were given in \cite{dc-as-lcd}.
In this paper we give new construction methods, by which we obtain many examples of these types of subspace codes.

We will follow the definition of a commutative association scheme given in \cite{BI}. 
Let $X$ be a finite set. A {\it commutative association scheme} with $d$ classes is a pair $(X,\mathcal{R})$ such that
\begin{enumerate}
 \item $\mathcal{R}= \{ R_0,R_1,\ldots ,R_d\}$ is a partition of $X \times X$,
 \item $R_0= \{ (x,x) \mid x \in X \} $,
 \item for any $i\in\{0,1,\ldots,d\}$, there exists $i'\in\{0,1,\ldots,d\}$ such that $R_i^\top:=\{(y,x)\mid (x,y)\in R_i\}=R_{i'}$,
 \item there are numbers $p_{ij}^k$ (the intersection numbers of the scheme) such that for any pair $(x,y) \in R_k$ the number of $z \in X$ such that $(x,z) \in R_i$ and $(z,y) \in R_j$ equals $p_{ij}^k$, 
 \item for any $i,j,k\in\{0,1,\ldots,d\}$, $p_{ij}^k=p_{ji}^k$.  
\end{enumerate} 
If $i'=i$ holds for any $i$, a commutative association scheme is said to be {\it symmetric}. 

The relations $R_i$, $i \in \{ 0,1, \ldots, d\}$, of a commutative association scheme determine the set of symmetric $(0,1)$-adjacency matrices $\mathcal{A}= \{A_0, A_1, \ldots, A_d \}$, such that $[A_i]_{xy}=1$ if and only if $(x,y) \in R_i$, which generate $(d+1)$-dimensional commutative and associative algebra over complex numbers called {\it the Bose-Mesner algebra} of the scheme. 
The matrices $\{A_0, A_1, \ldots , A_d \}$ satisfy the following equation
\begin{equation} \label{form1}
 A_i A_j=\sum_{k=0}^d p_{ij}^k A_k.
\end{equation}
Each matrix $A_i$, $i \in \{1,2, \ldots , d\}$, represents a simple directed graph $\Gamma_i$ on the set of vertices $X$, such that there is an arc from $x$ to $y$ in $\Gamma_i$ if and only if $(x,y)\in R_i$.

\section{Constructions for self-orthogonal and LCD subspace codes} \label{SO-LCD-subspace}

In this section, we provide a unified method to construct self-orthogonal and LCD subspace codes using  mutually quasi-unbiased weighing matrices, linked systems of symmetric designs, and linked systems of symmetric group divisible designs.

\begin{tm} \label{lcd-matrices}
Let $S=\{ M_1, M_2, \ldots, M_m \}$ be a set of integer square matrices of order $n$, and let $p$ be a prime number dividing the entries of $M_i M_j^\top$ for any $i,j\in\{1,2,\ldots,m\}$.   
Further, let $\mathbb{F}_q$ be the finite field of order $q$ with characteristic $p$. 
Then the following hold.
\begin{enumerate} 
\item The set of row spaces of the nonzero elements of the linear space spanned by the matrices $M_i$, $i=1,2, \ldots, m$, 
forms a self-orthogonal subspace code $\mathcal{C} \subseteq \mathbb F_q^{n}$.
\item The set of row spaces of the matrices  
$N_\alpha=\left[ \begin{array}{c|c}
X & \alpha I_n
\end{array} \right]$, $\alpha \in \mathbb F_{q} \setminus \{ 0 \}$,  
where $X$ is a nonzero element of the linear space spanned by the matrices $M_i$, $i=1,2, \ldots, m$, forms an LCD subspace code $\mathcal{C} \subseteq \mathbb F_q^{2n}$.
\end{enumerate}
\end{tm}
\begin{proof}
1: For a matrix $X=\sum_{i=1}^m a_iM_i$, $XX^\top =\sum_{i,j=1}^m a_ia_j M_iM_j^\top =0\pmod{p}$. 
Therefore it follows that $\mathcal{C}$ is a self-orthogonal subspace code.\\
2: Since $XX^\top=0\pmod{p}$, the matrix $N_\alpha N_\beta^\top=XX^\top+\alpha \beta I_n=\alpha \beta I_n\pmod{p}$ is a nonsingular matrix for $\alpha,\beta\in\mathbb{F}_q\setminus\{0\}$. It follows that the subspace determined by $N_\alpha$ intersects the dual of the subspace determined by $N_\beta$ trivially, so $\mathcal{C}$ is an LCD subspace code.
\end{proof}

We give examples of a set of matrices $\{M_1,\ldots,M_m\}$ satisfying the condition that the matrices $M_iM_j^\top$, $1\leq i,j\leq m$, have few entries and a prime number $p$ satisfying the assumption of Theorem~\ref{lcd-matrices}. 
\begin{ex}
A weighing matrix of order $n$ and weight $k$ is an $n\times n$ $(0,1,-1)$-matrix $W$ such that $WW^\top=kI$.  
Two weighing matrices $W_1$, $W_2$ of order $n$ and weight $k$ are said to be quasi-unbiased for parameters $(n,k,\ell,a)$ if there exist positive integers $a$, $\ell$ such that
$W_1W_2^{\top} = \sqrt{a} W$, where $W$ is a weighing matrix of order $n$ and weight $\ell$. If weighing matrices $W_1$ and $W_2$ are quasi-unbiased for parameters $(n,k,\ell,a)$, then $\ell = \frac{k^2}{a}$.
Weighing matrices $W_1, \ldots, W_m$ are said to be mutually quasi-unbiased weighing matrices (MQUWM) for parameters $(n,k,\ell,a)$ if any distinct two of them are quasi-unbiased for parameters $(n,k,\ell,a)$. 
Note that quasi-unbiased weighing matrices for parameters $(n,k,k,k)$ are known as \emph{unbiased weighing matrices of order $n$ and weight $k$}. 
The following parameters given in Table~\ref{tab:mquwm} except for $(n,n,n,n)$ are known to exist. In Table~\ref{tab:mquwm}, in the column marked with $p$, we give primes $p$ dividing $k$ and $\sqrt{a}$.

\begin{table}[H]
    \centering
    \begin{tabular}{c|c|c|c|c|c}
        $n$ &  $k$ & $\ell$ & $a$ & $p$ & Reference\\
        \hline
        $2^{2t+1}$ & $2^{2t+1}$ & $2^{2t}$ & $2^{2t+2}$ & $2$ & \cite{mquwm} \\
        $8,11,13,2d$ ($d\geq 5,d\neq 8$) & $4$ & $4$ & $4$ & $2$ & \cite{mquwm} \\
        $8$ & $8$ & $4$ & $16$ & $2$ & \cite{AHS}\\
        $12$ & $12$ & $9$ & $16$ & $2$ & \cite{AHS}\\
        $16$ & $16$ & $4$ & $64$ & $2$ & \cite{AHS}\\
        $24$ & $24$ & $4$ & $144$ & $2,3$ & \cite{AHS}\\
        $24$ & $24$ & $9$ & $64$ & $2$ & \cite{AHS}\\
        $32$ & $32$ & $4$ & $256$ & $2$ & \cite{AHS}\\
        $48$ & $48$ & $4$ & $576$ & $2,3$ & \cite{AHS}\\
        $48$ & $48$ & $9$ & $256$ & $2$ & \cite{AHS}\\
        $48$ & $48$ & $36$ & $64$ & $2$ & \cite{AHS}\\
    \end{tabular}
    \caption{Parameters of mutually quasi-unbiased weighing matrices.}
    \label{tab:mquwm}
\end{table}
\end{ex}

\begin{ex}
An \emph{orthogonal design of order $v$ and type $(s_1,\ldots,s_u)$}, detnoed $OD(v;s_1,\ldots,s_u)$, is a $v\times v$ $(0,\pm x_1,\ldots,\pm x_u)$ matrix $D$, where $x_1,\ldots,x_u$ are distinct commuting indeterminates, such that $DD^\top=(s_1x_1^2+\cdots+s_ux_u^2)I$. 

Two orthogonal designs $D_1,D_2$ of order $v$ and type $(s_1,\ldots,s_u)$ are \emph{unbiased with parameter $\alpha$} if 
\[
D_1D_2^\top=\frac{s_1x_1^2+\cdots+s_ux_u^2}{\sqrt{\alpha}}W, 
\]
where $W$ is a $(0,1,-1)$-matrix and $\alpha$ is a positive integer \cite{KS2018}. 
Note that $W$ is a weighing matrix of order $v$ and weight $\alpha$. 
Starting with an $OD(4;1,1,1,1)$ or $OD(4n;1,1,1,1)$, 
and using \cite[Proposition~3.11]{KS2018} with complete mutually suitable Latin squares of order a power of two, one can construct the following $f$ unbiased orthogonal designs over the indeterminates $x_1\ldots,x_4$ and $\alpha=1$ for any integer $k$; 
\begin{itemize}
\item $f=16^{2^k}$ and $OD(16^{2^k};4\cdot 16^{2^k-1},4\cdot 16^{2^k-1},4\cdot 16^{2^k-1},4\cdot 16^{2^k-1})$, 
\item $f=2^{2k-1}$ and $OD(4^{2k};4^{k},4^{k},4^{k},4^{k})$. 
\end{itemize}
By Lagrange's four square theorem, for any prime number $p$, there exist positive integers $n_1,\ldots,n_4$ such that $p=n_1^2+\cdots+n_4^2$. 
If we substitute $n_i$ into $x_i$, we have integer matrices $W_1,\ldots,W_{f}$ such that 
\[
W_iW_j^\top=pW_{i,j}
\]
where $W_{i,j}$ is a $(0,\pm1)$-matrix. 
Thus, the matrices $W_1,\ldots,W_{f}$ satisfy the assumption on Theorem~\ref{lcd-matrices}.  
\end{ex}

\begin{ex}
Let $v,k,\lambda$ be positive integers.  
A linked system of symmetric designs with parameters $(v,k,\lambda)$ is a set of $v\times v$ $(0,1)$-matrices $\{A_{i,j} \mid 1\leq i,j\leq f,i\neq j\}$ such that $A_{i,j}^\top=A_{j,i}$ and $A_{i,j}A_{j,i}=(k-\lambda)I_v+\lambda J_v$ for any distinct $i,j$ and that there exist integers $\sigma,\tau$ satisfying 
\begin{align*}
A_{i,j}A_{j,s}=\sigma A_{i,s}+\tau(J_v-A_{i,s})
\end{align*}
for any distinct $i,j,s$. 
Note that each $A_{i,j}$ is the incidence matrix of a symmetric $2$-$(v,k,\lambda)$ design. 
The parameters given in Table~\ref{tab:lsd} are known to exist. In Table~\ref{tab:lsd}, in the column marked with $p$, we give primes $p$ dividing $k,\lambda,\sigma,\tau$.

\begin{table}[H]
    \centering
    \begin{tabular}{c|c|c|c|c|c|c}
        $v$ &  $k$ &  $\lambda$ & $\sigma$ & $\tau$ & $p$ & Reference\\
        \hline
        $4n^2$ & $2n^2-n$ & $n^2-n$ & $n^2-\frac{n}{2}$ & $n^2-\frac{3n}{2}$ & 2 &\cite{KSlinked2} \\
    \end{tabular}
    \caption{Parameters of linked systems of symmetric designs}
    \label{tab:lsd}
\end{table}
\end{ex}

\begin{ex}
Let $v,k,m,n,\lambda_1,\lambda_2$ be positive integers with $v=mn$. 
A linked system of symmetric group divisible designs with parameters $(v,k,m,n,\lambda_1,\lambda_2)$ is a set of $v\times v$ $(0,1)$-matrices $\{A_{i,j} \mid 1\leq i,j\leq f,i\neq j\}$ such that $A_{i,j}^\top=A_{j,i}$ and $A_{i,j}A_{j,i}=kI_v+\lambda_1(I_m\otimes J_n-I_v)+\lambda_2 (J_v-I_m\otimes J_n)$ for any distinct $i,j$ and that there exist integers $\sigma,\tau$ satisfying 
\begin{align*}
A_{i,j}A_{j,s}=\sigma A_{i,s}+\tau(J_v-A_{i,s})
\end{align*}
for any distinct $i,j,s$. 
Note that each $A_{i,j}$ is the incidence matrix of a symmetric group divisible design with parameters $(v,k,m,n,\lambda_1,\lambda_2)$. 
The parameters given in Table~\ref{tab:lsgdd} are known to exist. In Table~\ref{tab:lsgdd}, in the column marked with $p$, we give primes $p$ dividing $k,\lambda_1,\lambda_2,\sigma,\tau$.

\begin{table}[H]
    \centering
    \begin{tabular}{c|c|c|c|c|c|c|c|c|c}
        $v$ &  $k$ & $m$ & $n$ & $\lambda_1$ & $\lambda_2$ & $\sigma$ & $\tau$  & $p$ & Reference\\
        \hline
        $56$ & $28$ & $7$ & $8$ & $12$ & $14$ &16 & 12& $2$  &\cite{KSlinked} \\
	 $108$ & $36$ & $36$ & $3$ & $0$ & $12$ &16 & 10& $2$ &\cite{KSlinked} \\
	 $132$ & $66$ & $11$ & $12$ & $30$ & $33$ &36 & 30& $3$ &\cite{KSlinked} \\
    \end{tabular}
    \caption{Parameters of linked systems of symmetric group divisible designs.}
    \label{tab:lsgdd}
\end{table}
\end{ex}

\begin{ex}
Let $v,k,m,n,\lambda_1,\lambda_2$ be positive integers with $v=mn$. 
A linked system of symmetric group divisible designs of type II with parameters $(v,k,m,n,\lambda_1,\lambda_2)$ is a set of $v\times v$ $(0,1)$-matrices $\{A_{i,j} \mid 1\leq i,j\leq f,i\neq j\}$ such that $A_{i,j}^\top=A_{j,i}$ and $A_{i,j}A_{j,i}=kI_v+\lambda_1(I_m\otimes J_n-I_v)+\lambda_2 (J_v-I_m\otimes J_n)$ for any distinct $i,j$ and that there exist integers $\sigma,\tau,\rho$ satisfying 
\begin{align*}
A_{i,j}A_{j,s}=\sigma A_{i,s}+\tau(J_v-A_{i,s}-I_m\otimes J_m)+\rho I_m\otimes J_m
\end{align*}
for any distinct $i,j,s$. 
Note that each $A_{i,j}$ is the incidence matrix of a symmetric group divisible design with parameters $(v,k,m,n,\lambda_1,\lambda_2)$. 
The parameters given in Table~\ref{tab:lsgdd2} are known to exist. In Table~\ref{tab:lsgdd2}, in the column marked with $p$, we give primes $p$ dividing $k,\lambda_1,\lambda_2,\sigma,\tau,\rho$.

\begin{table}[H]
    \centering
    \begin{tabular}{c|c|c|c|c|c|c|c|c|c|c}
        $v$ &  $k$ & $m$ & $n$ & $\lambda_1$ & $\lambda_2$ & $\sigma$ & $\tau$ & $\rho$ & $p$ & Reference\\
        \hline
	 $378$ & $117$ & $14$ & $27$ & $36$ & $36$ & 42 & 33 & 39 & 3 &\cite{KSlinked2} \\
    \end{tabular}
    \caption{Parameters of linked systems of symmetric group divisible designs of type II.}
    \label{tab:lsgdd2}
\end{table}
\end{ex}

\begin{ex}
In \cite[Theorem 5.8]{deza}, Deza digraphs of type II with adjacency matrices $N_{\alpha},\alpha\in \mathbb{F}_q$ was constructed for a prime power $q$ with characteristic $p$. 
These adjacency matrices satisfy that all the entries of $N_\alpha N_\beta^\top$ are divided by $q$ for any $\alpha,\beta\in\mathbb{F}_q$.  
Then the set of row spaces of the nonzero elements of the linear space spanned by the matrices $N_\alpha,\alpha\in \mathbb{F}_q$ forms a self-orthogonal subspace code $\mathcal{C} \subseteq \mathbb F_q^{q^2(2q+3)}$ \cite{dc-as-gc}.

In Appendix A, Deza graphs with adjacency matrices $N_\alpha$, $\alpha\in\mathbb{F}_q$, are constructed from the finite field $\mathbb{F}_q$ of order $q$ with odd characteristic $p$. The Deza graphs are a counterpart of the Deza digraphs of type II in 5 above to the undirected case.  
These adjacency matrices satisfy that all the entries of $N_\alpha N_\beta^\top$ are divided by $q$ for any $\alpha,\beta\in\mathbb{F}_q$.  
Then the set of row spaces of the nonzero elements of the linear space spanned by the matrices $N_\alpha,\alpha\in \mathbb{F}_q$ forms a self-orthogonal subspace code $\mathcal{C} \subseteq \mathbb F_q^{q^2(2q+3)}$.
\end{ex}

Similarly as in \cite[Theorem 5]{dc-as-lcd}, Theorem \ref{lcd-matrices} can be generalized for quotient matrices.
A partition of a square matrix $A$ is said to be equitable if all the blocks of the partitioned matrix have constant row sums and each of the diagonal blocks are of square order. A quotient matrix $Q$ of a square matrix $A$ corresponding to an equitable partition is a matrix whose entries are the constant row sums of the corresponding blocks of $A$ (see \cite{equitable}).

Let $\{ M_1, M_2, \ldots , M_m \}$ be a set of integer square matrices of order $n$, where the rows and columns of the matrices $\{ M_1, M_2, \ldots , M_m \}$ are indexed by the elements of the set 
$X= \{1,2, \ldots ,n \}$. An equitable partition of the set $\{ M_1, M_2, \ldots , M_m \}$ is a partition $\Pi= \{C_1, C_2, \ldots , C_{t} \}$ of $X$ which is equitable with respect to each of the matrices 
$M_i$, $i=1,2, \ldots ,m$. Let $C$ be the characteristic matrix of $\Pi$. Further, let a group $G$ act as an automorphism group of matrix $M_1,M_2,\ldots,M_m$. Then the $G$-orbits form an equitable partition of the matrices $M_1,M_2,\ldots,M_m$ and the corresponding quotient matrix $M'_i$ is called the orbit matrix of $M_i$ with respect to the action of $G$ (see \cite{orb-Had}). 

Let $M'_i$ denote the $t \times t$ quotient matrix of $M_i$ with respect to the partition $\Pi$. 
Then 
$$M_i'=(C^{\top}C)^{-1}C^{\top}M_i C$$
for each $i$. 
Since $CC^\top $ and each $M_i^\top$ commute, and $C^\top C$ is a diagonal matrix, 
\begin{align*}
M'_i(M'_j)^{\top}&=(C^{\top}C)^{-1}C^{\top}M_iC(C^{\top}M_j^\top C((C^{\top}C)^{-1})^\top)\\
&=(C^{\top}C)^{-1}C^{\top}M_iCC^{\top}M_j^\top C(C^{\top}C)^{-1}\\
&=(C^{\top}C)^{-1}C^{\top}M_iM_j^\top CC^{\top} C(C^{\top}C)^{-1}\\
&=(C^{\top}C)^{-1}C^{\top}M_iM_j^\top C. 
\end{align*}
Applying Theorem~\ref{lcd-matrices} to the quotient matrices, we obtain the following theorem. 

\begin{tm} \label{lcd_matrices-orb}
Let $S=\{ M_1, M_2, \ldots, M_m \}$ be a set of integer square matrices of order $n$, and let $p$ be a prime number dividing the entries of $M_i M_j^\top$ for any $i,j\in\{1,2,\ldots,m\}$.   
Further, let $\Pi$ be an equitable partition of the set $S$, with $t$ cells of the same length $\frac{n}{t}$, and let $M'_i$ denote the corresponding quotient matrix of $M_i$ with respect to 
$\Pi$, $i=1,2, \ldots, m$. 
Let $\mathbb{F}_q$ be the finite field of order $q$ with characteristic $p$. 
Then the following hold.
\begin{enumerate} 
\item 
The set of row spaces of the nonzero elements of the linear space spanned by the matrices $M'_1,M'_2, \ldots, M'_m$ forms a self-orthogonal subspace code $\mathcal{C} \subseteq \mathbb F_q^{t}$.
\item 
The set of row spaces of the matrices  
$N_x=\left[ \begin{array}{c|c}
X & \alpha_x I_t
\end{array} \right]$, $\alpha_x \in \mathbb F_{q} \setminus \{ 0 \}$,  
where $X$ is a nonzero element of the linear space spanned the matrices $M'_1,M'_2, \ldots, M'_m$, forms an LCD subspace code $\mathcal{C} \subseteq \mathbb F_q^{2t}$.
\end{enumerate}
\end{tm}

We give examples that illustrate a construction given in Theorem \ref{lcd_matrices-orb}.

\begin{ex}
Let $\mathcal{D}$ be a symmetric group divisible design with parameters $(v,k,m,n, \lambda_1, \lambda_2)$. The vertex partition from the definition of a symmetric group divisible design gives a partition (called the canonical partition) of the incidence matrix of $\mathcal{D}$. The canonical partition of a symmetric group divisible design is an equitable partition of its incidence matrix (see \cite{bose, ddgs}), and the corresponding quotient matrix satisfies the property
$$RR^{\top} = (k^2 - \lambda_2 v)I_m + \lambda_2 nJ_m.$$
Let $\{A_{i,j} \mid 1\leq i,j\leq f,i\neq j\}$ be a linked system of symmetric group divisible designs with parameters $(v,k,m,n,\lambda_1,\lambda_2)$ satisfying $A_{i,j}A_{j,s}=\sigma A_{i,s}+\tau(J_v-A_{i,s})$ for any distinct $i,j,s$. Further, let $M_{i,j}$, $1\leq i,j\leq f,i\neq j$, be the quotient matrices of $\{A_{i,j} \mid 1\leq i,j\leq f,i\neq j\}$ with respect to the canonical partition. 
If $p$ is a prime dividing $k,\lambda_1,\lambda_2,\sigma,\tau$ and $\mathbb{F}_q$ denotes the finite field of order $q$ with characteristic $p$, then 
\begin{enumerate} 
\item 
the set of row spaces of the nonzero elements of the linear space spanned by the matrices $\{M_{i,j} \mid 1\leq i,j\leq f,i\neq j\}$ forms a self-orthogonal subspace code $\mathcal{C} \subseteq \mathbb F_q^{m}$,
\item 
the set of row spaces of the matrices  
$N_\alpha=\left[ \begin{array}{c|c}
X & \alpha I_t
\end{array} \right]$, $\alpha \in \mathbb F_{q} \setminus \{ 0 \}$,  
where $X$ is a nonzero element of the linear space spanned the matrices $\{M_{i,j} \mid 1\leq i,j\leq f,i\neq j\}$, forms an LCD subspace code $\mathcal{C} \subseteq \mathbb F_q^{2m}$.
\end{enumerate}

In a similar way, subspace codes can be constructed from linked systems of symmetric group divisible designs of type II.
\end{ex}

\begin{ex}
Let $\mathcal{D}$ be a symmetric design with parameters $(v,k, \lambda)$ and $G$ be an automorphism group of $\mathcal{D}$. Assume that $G$ acts on the points and blocks of $\mathcal{D}$ in $t=\frac{v}{\omega}$ orbits of length $\omega$.
The orbit matrix $M$ of the symmetric design $\mathcal{D}$ with respect to the action of $G$ satisfies
\begin{align*}
\sum_{j=1}^t \gamma_{ij}&=k,  \\
\sum_{i=1}^t \gamma_{ij}\gamma_{is}&= \lambda \omega + \delta_{js}\cdot(k- \lambda), 
\end{align*}
where $\gamma_{ij}$ is the entry of $M$ corresponding to the $i$th block orbit and the $j$th point orbit, and $\delta_{js}$ is the Kronecker delta (see \cite{har_ton, zj}).

Let $\{A_{i,j} \mid 1\leq i,j\leq f,i\neq j\}$ be a linked system of symmetric designs with parameters $(v,k,\lambda)$, such that $A_{i,j}A_{j,s}=\sigma A_{i,s}+\tau(J_v-A_{i,s})$ for any distinct $i,j,s$,
where $A_{i,j}$ is the incidence matrix of a symmetric design $\mathcal{D}_{i,j}$ with parameters $(v,k, \lambda)$. Suppose that a group $G$ act as an automorphism group on each design $\mathcal{D}_{i,j}$,
for $1\leq i,j\leq f,i\neq j$, in $t$ orbits of length $\omega$. Let $p$ be a prime dividing $k,\lambda_1,\lambda_2,\sigma,\tau$, and $\mathbb{F}_q$ be the finite field of order $q$ with characteristic $p$.
Then the corresponding orbit matrices $M_{i,j}$, $1\leq i,j\leq f,i\neq j$, can be used for a construction of self-orthogonal subspace codes in $\mathbb F_q^{t}$ and LCD subspace codes in $\mathbb F_q^{2t}$, respectively, as given in Theorem \ref{lcd_matrices-orb}.
\end{ex}

\begin{ex}
Let $S=\{ W_1,W_2, \ldots , W_m \}$ be mutually quasi-unbiased weighing matrices for parameters $(n,k,l,a)$.
Further, let $\Pi$ be an equitable partition of the set $S$, with $t$ cells of the same length $\frac{n}{t}$, and let $M_i$ denote the corresponding quotient matrix of $W_i$ with respect to 
$\Pi$, $i=1,2, \ldots, m$. Let $p$ be a prime number dividing $k$ and $\sqrt{a}$.
Then the matrices $M_1,M_2, \ldots , M_m$ can be used for a construction of a self-orthogonal subspace code in $\mathbb F_q^{t}$ and an LCD subspace code in $\mathbb F_q^{2t}$, respectively, for some positive integer $r$ after $q = p^r$.
\end{ex}

In \cite{dsa-as}, the following construction of self-orthogonal subspace codes from association schemes was given and applied to quotient matrices of distance-regular graphs.

\begin{kor} \label{so_subcode}
Let $\Pi$ be an equitable partition of a $d$-class commutative association scheme $(X,\mathcal{R})$ with $n$ cells of the same length $\frac{|X|}{n}$ and let $p$ be a prime number.
Let $\mathcal{A}= \{A_0, A_1, \dots, A_d \}$ be the set of adjacency matrices of $(X,\mathcal{R})$, and let $M_i$ denote the corresponding quotient matrix of $A_i$ with respect to $\Pi$.
Further, let $I=\{i_1, i_2, \ldots , i_t\} \subseteq \{0, 1, \ldots , d \}$ and $p|p_{x,y'}^k$, for all $k \in \{0,1,\ldots ,d\}$ and all $x, y \in I$.
Then the set of row spaces of nonzero elements of the linear space generated by the matrices $M_i$, $i \in I$, forms a self-orthogonal subspace code $\mathcal{C} \subseteq \mathbb F_q^n$,
where $q=p^m$ is a prime power.
\end{kor}

Since an adjacency matrix can be seen as an orbit matrix with respect to the trivial group, Corollary~\ref{so_subcode} can be applied to adjacency matrices of association schemes.
We give example of association schemes satisfying the assumption of Corollary~\ref{so_subcode}. 

\begin{ex}
A \emph{Bush-type Hadamard matrix} of order $4n^2$ is a block matrix $H= [H_{ij}]$ with block size $2n$, $H_{ii}=J_{2n}$ and $H_{ij}J_{2n}=J_{2n}H_{ij}=0$, $i \neq j$, $1 \le i \le 2n$, $1 \le j \le 2n$,
where $J_{2n}$ is the $2n$ by $2n$ matrix of all 1 entries. Obviously, Bush-type Hadamard matrices are regular i.e. they have a property that the row sums and column sums are all the same.

In \cite[Theorem 5]{MUH-BT}, a $5$-class symmetric association scheme was obtained from $m$ mutually unbiased Bush-type Hadamard matrices of order $n$ with the intersection matrices $B_i=(p_{ij}^k)_{j,k=0}^5$ given as follows:
\begin{align*}
B_1&=\left(
\begin{smallmatrix}
 0 & 1 & 0 & 0 & 0 & 0 \\
 2 n-1 & 2 (n-1) & 0 & 0 & 0 & 0 \\
 0 & 0 & 2 n-1 & 0 & 0 & 0 \\
 0 & 0 & 0 & 2 n-1 & 0 & 0 \\
 0 & 0 & 0 & 0 & n-1 & n \\
 0 & 0 & 0 & 0 & n & n-1 \\
\end{smallmatrix}
\right),\displaybreak[0]\\
B_2&=\left(
\begin{smallmatrix}
 0 & 0 & 1 & 0 & 0 & 0 \\
 0 & 0 & 2 n-1 & 0 & 0 & 0 \\
 2 n (2 n-1) & 2 n (2 n-1) & 4 (n-1) n & 0 & 0 & 0 \\
 0 & 0 & 0 & 0 & 2 n & 2 n \\
 0 & 0 & 0 & n (2 n-1) & 2 (n-1) n & 2 (n-1) n \\
 0 & 0 & 0 & n (2 n-1) & 2 (n-1) n & 2 (n-1) n \\
\end{smallmatrix}
\right),\displaybreak[0]\\
B_3&=\left(
\begin{smallmatrix}
 0 & 0 & 0 & 1 & 0 & 0 \\
 0 & 0 & 0 & 2 n-1 & 0 & 0 \\
 0 & 0 & 0 & 0 & 2 n & 2 n \\
 2 m n & 2 m n & 0 & 2 (m-1) n & 0 & 0 \\
 0 & 0 & m n & 0 & (m-1) n & (m-1) n \\
 0 & 0 & m n & 0 & (m-1) n & (m-1) n \\
\end{smallmatrix}
\right),\displaybreak[0]\\
B_4&=\left(
\begin{smallmatrix}
 0 & 0 & 0 & 0 & 1 & 0 \\
 0 & 0 & 0 & 0 & n-1 & n \\
 0 & 0 & 0 & n (2 n-1) & 2 (n-1) n & 2 (n-1) n \\
 0 & 0 & m n & 0 & (m-1) n & (m-1) n \\
 m n (2 n-1) & m (n-1) n & m (n-1) n & \frac{1}{2} (m-1) n (2 n-1) & \frac{1}{2} (m-1) n (2 n-1) & \frac{1}{2} (m-1) n (2 n-3) \\
 0 & m n^2 & m (n-1) n & \frac{1}{2} (m-1) n (2 n-1) & \frac{1}{2} (m-1) n (2 n-3) & \frac{1}{2} (m-1) n (2 n-1) \\
\end{smallmatrix}
\right),\\
B_5&=\left(
\begin{smallmatrix}
 0 & 0 & 0 & 0 & 0 & 1 \\
 0 & 0 & 0 & 0 & n & n-1 \\
 0 & 0 & 0 & n (2 n-1) & 2 (n-1) n & 2 (n-1) n \\
 0 & 0 & m n & 0 & (m-1) n & (m-1) n \\
 0 & m n^2 & m (n-1) n & \frac{1}{2} (m-1) n (2 n-1) & \frac{1}{2} (m-1) n (2 n-3) & \frac{1}{2} (m-1) n (2 n-1) \\
 m n (2 n-1) & m (n-1) n & m (n-1) n & \frac{1}{2} (m-1) n (2 n-1) & \frac{1}{2} (m-1) n (2 n-1) & \frac{1}{2} (m-1) n (2 n-3) \\
\end{smallmatrix}
\right)
\end{align*}
Let $p$ be a prime number dividing $\frac{n}{2}$.
Then the set of row spaces of the nonzero elements of the matrix algebra generated by the matrices $A_2,A_3,A_4,A_5$ forms a self-orthogonal subspace code $\mathcal{C} \subseteq \mathbb F_q^{4n^2(m+1)}$, 
for a prime power $q$ with characteristic $p$.
\end{ex}



\begin{ex}
In \cite[Theorem 10]{MUH-BT}, an $8$-class symmetric association scheme was obtained from $m$ mutually unbiased Bush-type Hadamard matrices of order $n$ with the intersection matrices $B_i=(p_{ij}^k)_{j,k=0}^8$ given as follows:
\begin{align*}
B_1&=\left(
\begin{smallmatrix}
 1 & 2 n-1 & 2 n-1 & 4 \left(2 n^2-n\right) & 2 m n & 2 m n & 2 m n (2 n-1) & 2 \left(2 m n^2-m n\right) & 1 \\
 1 & -1 & 1 & 0 & 0 & 0 & 2 m n & -2 m n & -1 \\
 1 & 2 n-1 & 1-2 n & 0 & 2 m n & -2 m n & 0 & 0 & -1 \\
 1 & 2 n-1 & 2 n-1 & -4 n & 2 m n & 2 m n & -2 m n & -2 m n & 1 \\
 1 & -1 & -1 & 0 & 0 & 0 & 0 & 0 & 1 \\
 1 & 2 n-1 & 2 n-1 & -4 n & -2 n & -2 n & 2 n & 2 n & 1 \\
 1 & 2 n-1 & 1-2 n & 0 & -2 n & 2 n & 0 & 0 & -1 \\
 1 & -1 & 1 & 0 & 0 & 0 & -2 n & 2 n & -1 \\
 1 & 2 n-1 & 2 n-1 & 4 \left(2 n^2-n\right) & -2 n & -2 n & -2 n (2 n-1) & -2 \left(2 n^2-n\right) & 1 \\
\end{smallmatrix}
\right),\displaybreak[0]\\
B_2&=\left(
\begin{smallmatrix}
 0 & 1 & 0 & 0 & 0 & 0 & 0 & 0 & 0 \\
 2 n-1 & 2 (n-1) & 0 & 0 & 0 & 0 & 0 & 0 & 0 \\
 0 & 0 & 2 (n-1) & 0 & 0 & 0 & 0 & 0 & 2 n-1 \\
 0 & 0 & 0 & 2 n-1 & 0 & 0 & 0 & 0 & 0 \\
 0 & 0 & 0 & 0 & 2 n-1 & 0 & 0 & 0 & 0 \\
 0 & 0 & 0 & 0 & 0 & 2 n-1 & 0 & 0 & 0 \\
 0 & 0 & 0 & 0 & 0 & 0 & n-1 & n & 0 \\
 0 & 0 & 0 & 0 & 0 & 0 & n & n-1 & 0 \\
 0 & 0 & 1 & 0 & 0 & 0 & 0 & 0 & 0 \\
\end{smallmatrix}
\right),\displaybreak[0]\\
B_3&=\left(
\begin{smallmatrix}
 0 & 0 & 1 & 0 & 0 & 0 & 0 & 0 & 0 \\
 0 & 0 & 2 (n-1) & 0 & 0 & 0 & 0 & 0 & 2 n-1 \\
 2 n-1 & 2 (n-1) & 0 & 0 & 0 & 0 & 0 & 0 & 0 \\
 0 & 0 & 0 & 2 n-1 & 0 & 0 & 0 & 0 & 0 \\
 0 & 0 & 0 & 0 & 0 & 2 n-1 & 0 & 0 & 0 \\
 0 & 0 & 0 & 0 & 2 n-1 & 0 & 0 & 0 & 0 \\
 0 & 0 & 0 & 0 & 0 & 0 & n & n-1 & 0 \\
 0 & 0 & 0 & 0 & 0 & 0 & n-1 & n & 0 \\
 0 & 1 & 0 & 0 & 0 & 0 & 0 & 0 & 0 \\
\end{smallmatrix}
\right),\displaybreak[0]\\
B_4&=\left(
\begin{smallmatrix}
 0 & 0 & 0 & 0 & 1 & 0 \\
 0 & 0 & 0 & 0 & n-1 & n \\
 0 & 0 & 0 & n (2 n-1) & 2 (n-1) n & 2 (n-1) n \\
 0 & 0 & m n & 0 & (m-1) n & (m-1) n \\
 m n (2 n-1) & m (n-1) n & m (n-1) n & \frac{1}{2} (m-1) n (2 n-1) & \frac{1}{2} (m-1) n (2 n-1) & \frac{1}{2} (m-1) n (2 n-3) \\
 0 & m n^2 & m (n-1) n & \frac{1}{2} (m-1) n (2 n-1) & \frac{1}{2} (m-1) n (2 n-3) & \frac{1}{2} (m-1) n (2 n-1) \\
\end{smallmatrix}
\right),\displaybreak[0]\\
B_5&=\left(
\begin{smallmatrix}
 0 & 0 & 0 & 1 & 0 & 0 & 0 & 0 & 0 \\
 0 & 0 & 0 & 2 n-1 & 0 & 0 & 0 & 0 & 0 \\
 0 & 0 & 0 & 2 n-1 & 0 & 0 & 0 & 0 & 0 \\
 4 n (2 n-1) & 4 n (2 n-1) & 4 n (2 n-1) & 8 (n-1) n & 0 & 0 & 0 & 0 & 4 n (2 n-1) \\
 0 & 0 & 0 & 0 & 0 & 0 & 2 n & 2 n & 0 \\
 0 & 0 & 0 & 0 & 0 & 0 & 2 n & 2 n & 0 \\
 0 & 0 & 0 & 0 & 2 n (2 n-1) & 2 n (2 n-1) & 4 (n-1) n & 4 (n-1) n & 0 \\
 0 & 0 & 0 & 0 & 2 n (2 n-1) & 2 n (2 n-1) & 4 (n-1) n & 4 (n-1) n & 0 \\
 0 & 0 & 0 & 1 & 0 & 0 & 0 & 0 & 0 \\
\end{smallmatrix}
\right)
\displaybreak[0]\\
B_6&=\left(
\begin{smallmatrix}
 0 & 0 & 0 & 0 & 1 & 0 & 0 & 0 & 0 \\
 0 & 0 & 0 & 0 & 2 n-1 & 0 & 0 & 0 & 0 \\
 0 & 0 & 0 & 0 & 0 & 2 n-1 & 0 & 0 & 0 \\
 0 & 0 & 0 & 0 & 0 & 0 & 2 n & 2 n & 0 \\
 2 m n & 2 m n & 0 & 0 & 2 (m-1) n & 0 & 0 & 0 & 0 \\
 0 & 0 & 2 m n & 0 & 0 & 2 (m-1) n & 0 & 0 & 2 m n \\
 0 & 0 & 0 & m n & 0 & 0 & (m-1) n & (m-1) n & 0 \\
 0 & 0 & 0 & m n & 0 & 0 & (m-1) n & (m-1) n & 0 \\
 0 & 0 & 0 & 0 & 0 & 1 & 0 & 0 & 0 \\
\end{smallmatrix}
\right)
\displaybreak[0]\\
B_7&=\left(
\begin{smallmatrix}
 0 & 0 & 0 & 0 & 0 & 1 & 0 & 0 & 0 \\
 0 & 0 & 0 & 0 & 0 & 2 n-1 & 0 & 0 & 0 \\
 0 & 0 & 0 & 0 & 2 n-1 & 0 & 0 & 0 & 0 \\
 0 & 0 & 0 & 0 & 0 & 0 & 2 n & 2 n & 0 \\
 0 & 0 & 2 m n & 0 & 0 & 2 (m-1) n & 0 & 0 & 2 m n \\
 2 m n & 2 m n & 0 & 0 & 2 (m-1) n & 0 & 0 & 0 & 0 \\
 0 & 0 & 0 & m n & 0 & 0 & (m-1) n & (m-1) n & 0 \\
 0 & 0 & 0 & m n & 0 & 0 & (m-1) n & (m-1) n & 0 \\
 0 & 0 & 0 & 0 & 1 & 0 & 0 & 0 & 0 \\
\end{smallmatrix}
\right)
\\
B_8&=\left(
\begin{smallmatrix}
 0 & 0 & 0 & 0 & 0 & 0 & 1 & 0 & 0 \\
 0 & 0 & 0 & 0 & 0 & 0 & n-1 & n & 0 \\
 0 & 0 & 0 & 0 & 0 & 0 & n & n-1 & 0 \\
 0 & 0 & 0 & 0 & 2 n (2 n-1) & 2 n (2 n-1) & 4 (n-1) n & 4 (n-1) n & 0 \\
 0 & 0 & 0 & m n & 0 & 0 & (m-1) n & (m-1) n & 0 \\
 0 & 0 & 0 & m n & 0 & 0 & (m-1) n & (m-1) n & 0 \\
 2 m n (2 n-1) & 2 m (n-1) n & 2 m n^2 & 2 m (n-1) n & (m-1) n (2 n-1) & (m-1) n (2 n-1) & (m-1) n (2 n-1) & (m-1) n (2 n-3) & 0 \\
 0 & 2 m n^2 & 2 m (n-1) n & 2 m (n-1) n & (m-1) n (2 n-1) & (m-1) n (2 n-1) & (m-1) n (2 n-3) & (m-1) n (2 n-1) & 2 m n (2 n-1) \\
 0 & 0 & 0 & 0 & 0 & 0 & 0 & 1 & 0 \\
\end{smallmatrix}
\right)
\end{align*}
Let $p$ be a prime number dividing $n$.
Then the set of row spaces of the nonzero elements of the matrix algebra generated by the matrices $A_3,A_4,A_5,A_6,A_7$ forms a self-orthogonal subspace code $\mathcal{C} \subseteq \mathbb F_q^{8n^2(m+1)}$, 
for a prime power $q$ with characteristic $p$.
\end{ex}

The next example comes from a commutative association scheme. 
\begin{ex}
in \cite[Theorem 5.3 (ii)]{KS}, a $4$-class commutative association scheme on $n(n-1)$ vertices was obtained from a Hadamard matrix of order $n\geq4$ with the intersection matrices $B_i=(p_{ij}^k)_{j,k=0}^4$ given as follows:
\begin{align*}
B_1&=\left(
\begin{smallmatrix}
 0 & 1 & 0 & 0 & 0 \\
 0 & \frac{1}{4} (n-3) n & \frac{1}{4} (n-3) n & \frac{n^2}{4} & \frac{1}{4} (n-2) n \\
 \frac{1}{2} (n-2) n & \frac{1}{4} (n-3) n & \frac{1}{4} (n-3) n & \frac{1}{4} (n-4) n & \frac{1}{4} (n-2) n \\
 0 & \frac{n-4}{4} & \frac{n}{4} & 0 & 0 \\
 0 & \frac{n}{4} & \frac{n}{4} & 0 & 0 \\
\end{smallmatrix}
\right),\displaybreak[0]\\
B_2&=\left(
\begin{smallmatrix}
 0 & 0 & 1 & 0 & 0 \\
 \frac{1}{2} (n-2) n & \frac{1}{4} (n-3) n & \frac{1}{4} (n-3) n & \frac{1}{4} (n-4) n & \frac{1}{4} (n-2) n \\
 0 & \frac{1}{4} (n-3) n & \frac{1}{4} (n-3) n & \frac{n^2}{4} & \frac{1}{4} (n-2) n \\
 0 & \frac{n}{4} & \frac{n-4}{4} & 0 & 0 \\
 0 & \frac{n}{4} & \frac{n}{4} & 0 & 0 \\
\end{smallmatrix}
\right),\displaybreak[0]\\
B_3&=\left(
\begin{smallmatrix}
 0 & 0 & 0 & 1 & 0 \\
 0 & \frac{n-4}{4} & \frac{n}{4} & 0 & 0 \\
 0 & \frac{n}{4} & \frac{n-4}{4} & 0 & 0 \\
 \frac{n-2}{2} & 0 & 0 & \frac{n-4}{2} & 0 \\
 0 & 0 & 0 & 0 & \frac{n-2}{2} \\
\end{smallmatrix}
\right),\displaybreak[0]\\
B_4&=\left(
\begin{smallmatrix}
 0 & 0 & 0 & 0 & 1 \\
 0 & \frac{n}{4} & \frac{n}{4} & 0 & 0 \\
 0 & \frac{n}{4} & \frac{n}{4} & 0 & 0 \\
 0 & 0 & 0 & 0 & \frac{n-2}{2} \\
 \frac{n}{2} & 0 & 0 & \frac{n}{2} & 0 \\
\end{smallmatrix}
\right).\displaybreak[0]
\end{align*}
Let $p$ be a prime number dividing $n/4$.
Then the set of row spaces of the nonzero elements of the matrix algebra generated by the matrices $A_1,A_2$ forms a self-orthogonal subspace code $\mathcal{C} \subseteq \mathbb F_q^{n(n-1)}$, 
for a prime power $q$ with characteristic $p$.
\end{ex}

%
%
%

\section{Statements and Declarations}

\subsection{Funding}
Dean Crnkovi\'c and Andrea \v Svob were supported by {\rm C}roatian Science Foundation under the projects 4571 and 5713, respectively.  
Hadi Kharaghani is supported by an NSERC Discovery Grant. 
Sho Suda was supported by JSPS KAKENHI Grant Number 22K03410.

\subsection{Availability of data and material}
Not applicable.


\appendix
\def\thesection{Appendix \Alph{section}}

\section{Construction of Deza graphs form finite fields}

A $k$-regular graph on $n$ vertices is called a {\it Deza graph} with parameters $(n,k,b,a)$ if the number of common neighbors of two distinct vertices is either $a$ or $b$ ($a\leq b$).  

In this section, we give a construction of a Deza graph $D$ with parameters $(q^2(2q+3), 2q(q+1),3q,2q)$, where $q$ is a prime power.
Furthermore, we give a decomposition of the complete graph $K_{q^2(2q+3)}$ into $(q-1)$ Deza graphs, each of which is isomorphic to the Deza graph $D$.

We denote by $\mathbb{F}_q$ the finite field of $q=p^m$ elements, where $p$ is a prime number and $m$ is a positive integer.
Let $U$ be a circulant matrix of order $p$ with first row $(0,1,0,\ldots,0)$, and let $R$ be the back-identity matrix of order $p$.
Regarding the additive group of $\mathbb{F}_{q}$ as $\mathbb{F}_p^m$, we define a family of auxiliary matrices $C_{a,\alpha}$ of order $q^2$ for $a=(a_1,a_2,\dots,a_m) \in \mathbb{F}_q \cup \{x,y\}$ and $\alpha=(\alpha_1,\alpha_2,\dots,\alpha_m) \in \mathbb{F}_q$.
\begin{align*}
    C_{a,\alpha} &= 
    \begin{cases}
        O_{q^2} & \textrm{if } a=x, \\
        (\otimes_{i=1}^m U^{\alpha_i} R) \otimes J_q & \textrm{if } a=y,\\
        \sum_{c=(c_1,c_2,\dots,c_m) \in \mathbb{F}_q} (\otimes_{i=1}^m U^{c_i} R) \otimes (\otimes_{i=1}^m U^{a_ic_i + \alpha_i} R) & \textrm{if } a \in \mathbb{F}_q.
    \end{cases}
\end{align*}
In a similar manner, we define another family of auxiliary matrices $D_{a, \alpha}$ of order $q^2$ for $a=(a_1,a_2,\dots,a_m) \in \mathbb{F}_q \cup \{x,y\}$ and $\alpha=(\alpha_1,\alpha_2,\dots,\alpha_m) \in \mathbb{F}_q$.
\begin{align*}
    D_{a,\alpha} &= 
    \begin{cases}
        O_{q^2} & \textrm{if } a=x, \\
        (\otimes_{i=1}^m U^{\alpha_i}) \otimes J_q & \textrm{if } a=y, \\
        \sum_{c=(c_1,c_2,\dots,c_m) \in \mathbb{F}_q} (\otimes_{i=1}^m U^{c_i}) \otimes (\otimes_{i=1}^m U^{a_i c_i + \alpha_i}) & \textrm{if } a \in \mathbb{F}_q.
    \end{cases}
\end{align*}
Let $L$ be a circulant matrix of order $2q+3$ with the first row $(a_i)_{i=1}^{2q+3}$ satisfying 
$$
a_1=x,a_2=y, \{a_i : 2\leq i \leq q+1\}=\mathbb{F}_q, a_{2q+4-i}=a_{i+1} \text{ for }i\in\{2,\ldots,2q+1\}. 
$$
Note that we can write $L$ as $L=\sum_{a \in \mathbb{F}_q \cup \{x,y\}} a\cdot P_a$, where $P_a$ is a permutation matrix.
Let us fix a bijection $\varphi:\mathbb{F}_q\cup\{y\}\rightarrow \{1,\ldots,q+1\}$ such that $P_a=V^{\varphi(a)}+V^{-\varphi(a)}$ where $V$ is the shift matrix of order $2q+3$.
For $\alpha \in \mathbb{F}_q$, we define a $(0,1)$-matrix $N_{\alpha}$ of order $q^2(2q+3)$ to be
\begin{align*}
N_\alpha = \sum_{a\in \mathbb{F}_q\cup\{y\}} P_a\otimes C_{a,\alpha}. 
\end{align*}

In the following paragraphs, we will see that $N_{\alpha}$ is an adjacency matrix of a Deza graph.
To achieve this goal, we prepare some lemmas on the auxiliary matrices $C_{a,\alpha}$ and $D_{a,\alpha}$.

\begin{lema} \label{lem:UR}
    Let $a \in \mathbb{F}_p$.
    Then the following holds:
    \begin{enumerate}
        \item $RU^aR=U^{-a}$.
        \item $U^aR$ is symmetric.
    \end{enumerate}
\end{lema}
\begin{proof}
    If $p=2$, then $R=U$ and $a \in \{0,1\}$, so the assertion clearly holds.
    From now on, we assume that $p$ is odd.
    The matrices $U$ and $R$ represent a cyclic shift $\tau=(0,1,\dots,p-1)$ and an involution $\sigma=(0,p-1)(1,p-2)\dots((p-3)/2,(p+1)/2)$, respectively.
    This immediately gives $U^{-1}=U^T$ and $R^{-1}=R$.
    Furthermore, we have $\sigma\tau\sigma\tau(\alpha)=\sigma\tau\sigma(\alpha+1)=\sigma\tau(-(\alpha+1))=\sigma(-\alpha)=\alpha$ for any $\alpha \in \mathbb{F}_q$, which shows $RUR=U^{-1}$.
    Now, the first part of the lemma follows by $RU^aR=(RUR)(RUR)\cdots(RUR)=U^{-a}$.

    To prove the second part of the lemma, we apply the first part and observe $(U^aR)^T=R^T(U^a)^T=RU^{-a}=R(RU^aR)=U^aR$.
    This completes the proof.
\end{proof}

\begin{lema}\label{lem:CD}
    For $a \in \mathbb{F}_q \cup \{y\}$ and $\alpha, \beta \in \mathbb{F}_q$, the following holds:
    \begin{enumerate}
        \item $C^T_{a,\alpha}=C_{a,\alpha}$ and $D^T_{a,\alpha}=D_{a,-\alpha}$.
        \item
                $
                C_{a,\alpha} C_{b,\beta} =
                \begin{cases}
                    q D_{a,\alpha-\beta} & \textrm{if } a=b,\\
                    J_{q^2} & \textrm{otherwise}.
                \end{cases}
                $
        \item$\sum_{\alpha \in \mathbb{F}_q} C_{a, \alpha} = J_{q^2}$.
        \item$\sum_{a \in \mathbb{F}_q \cup \{y\}} D_{a,\alpha} = q I_q \otimes (\otimes_{i=1}^m U^{\alpha_i}) + (\otimes_{i=1}^m (J_q - I_q + U^{\alpha_i})) \otimes J_{q}$.
        \item $\sum_{a,b\in \mathbb{F}_q\cup\{y\},a\neq b} P_{a}P_{b}=2q(J_{2q+3}-I_{2q+3})$.
    \end{enumerate}
\end{lema}
\begin{proof}
\begin{enumerate}
    \item
        We only give a proof for the case $a \neq y$, since an analogous argument proves the assertion for $a = y$.
        By Lemma~\ref{lem:UR}, we have
        \begin{align*}
            C_{a, \alpha}^T
            &=\sum_{c \in \mathbb{F}_q} (\otimes_{i=1}^m U^{c_i} R)^T \otimes (\otimes_{i=1}^m U^{a_i c_i +\alpha_i}R)^T \\
            &=\sum_{c \in \mathbb{F}_q} (\otimes_{i=1}^m U^{c_i} R) \otimes (\otimes_{i=1}^m U^{a_i c_i +\alpha_i}R) \\
            &=C_{a, \alpha},
        \end{align*}
        and
        \begin{align*}
            D_{a,\alpha}^T
            &=\sum_{c \in \mathbb{F}_q} (\otimes_{i=1}^m U^{c_i})^T \otimes (\otimes_{i=1}^m U^{a_i c_i + \alpha_i})^T \\
            &=\sum_{c \in \mathbb{F}_q} (\otimes_{i=1}^m U^{-c_i}) \otimes (\otimes_{i=1}^m U^{a_i (-c_i) - \alpha_i}) \\
            &= D_{a,-\alpha},
        \end{align*}
        as required.
    
    \item
        First, we assume that $a \neq y$ and $b \neq y$, and compute the product $C_{a,\alpha} C_{b,\alpha}$ as follows:
        \begin{align*}
            C_{a,\alpha} C_{b,\beta} 
            &= (\sum_{c \in \mathbb{F}_q} (\otimes_{i=1}^m U^{c_i} R) \otimes (\otimes_{i=1}^m U^{a_i c_i +\alpha_i}R)) 
                (\sum_{d \in \mathbb{F}_q} (\otimes_{i=1}^m U^{d_i} R) \otimes (\otimes_{i=1}^m U^{b_i d_i +\beta_i}R)) \nonumber \\
            &= \sum_{c \in \mathbb{F}_q} \sum_{d \in \mathbb{F}_q} ((\otimes_{i=1}^m U^{c_i}R U^{d_i}R) \otimes
                (\otimes_{i=1}^m U^{a_i c_i+\alpha_i}R U^{b_i d_i+\beta_i}R)) \\
            &= \sum_{c \in \mathbb{F}_q} \sum_{d \in \mathbb{F}_q} ((\otimes_{i=1}^m U^{c_i-d_i}) \otimes (\otimes_{i=1}^m U^{a_i c_i-b_i d_i + (\alpha_i - \beta_i)}))\\
            &= \sum_{n \in \mathbb{F}_q} ((\otimes_{i=1}^m U^{n_i}) \otimes (\otimes_{i=1}^m \sum_{k \in \mathbb{F}_q} U^{a_i n_i+k_i(a_i-b_i)+ (\alpha_i - \beta_i)}) \\ 
            &= \begin{cases}
                \sum_{n \in \mathbb{F}_q} ((\otimes_{i=1}^m U^{n_i}) \otimes J_q) & \textrm{if } a \neq b, \\
                \sum_{n \in \mathbb{F}_q} ((\otimes_{i=1}^m U^{n_i}) \otimes (q \otimes_{i=1}^m U^{a_i n_i + (\alpha_i - \beta_i)}) & \textrm{if } a = b.
            \end{cases} \\
            &= \begin{cases}
                J_{q^2} & \textrm{if } a \neq b, \\
                q D_{a, \alpha-\beta} & \textrm{if } a = b.
            \end{cases}
        \end{align*}

        Secondly, we assume that $a = y$ and $b \neq y$.
        Then we have
        \begin{align*}
        C_{y,\alpha} C_{b,\beta}
        &= ((\otimes_{i=1}^m U^{\alpha_i}R) \otimes J_q) \sum_{c \in \mathbb{F}_q} ((\otimes_{i=1}^m U^{c_i}) \otimes (\otimes_{i=1}^m U^{b_ic_i+\beta_i})) \\
        &= \sum_{c \in \mathbb{F}_q} ((\otimes_{i=1}^m U^{\alpha_i} R U^{c_i}) \otimes (J_q \otimes_{i=1}^m U^{b_ic_i+\beta_i}) \\
        &= \sum_{c \in \mathbb{F}_q} ((\otimes_{i=1}^m U^{\alpha_i-c_i}) \otimes J_q)\\
        &= J_{q^2}.
        \end{align*}

        Now, we conclude our proof by the following computation:
        \begin{align*}
        C_{y,\alpha} C_{y,\beta}
        &= (\otimes_{i=1}^m U^{\alpha_i}R U^{\beta_i}R) \otimes J_q J_q \\
        &= (\otimes_{i=1}^m U^{\alpha_i - \beta_i}) \otimes q J_q \\
        &= q D_{y, \alpha-\beta}.
        \end{align*}
        
    \item
        By bilinearity of the tensor product, we have
        \begin{align*}
            \sum_{\alpha \in \mathbb{F}_q} C_{a, \alpha}
            &= \sum_{\alpha \in \mathbb{F}_q} \sum_{c \in \mathbb{F}_q} ((\otimes_{i=1}^m U^{c_i}) \otimes (\otimes_{i=1}^m U^{a_ic_i+\alpha_i})) \\
            &= \sum_{c \in \mathbb{F}_q} ((\otimes_{i=1}^m U^{c_i}) \otimes \sum_{\alpha \in \mathbb{F}_q} (\otimes_{i=1}^m U^{a_ic_i+\alpha_i})) \\
            &= \sum_{c \in \mathbb{F}_q} ((\otimes_{i=1}^m U^{c_i}) \otimes J_{q}) \\
            &= J_{q^2},
        \end{align*}
        as required.
        
    \item
        By bilinearity of the tensor product, we have
        \begin{align*}
            \sum_{a \in \mathbb{F}_q \cup \{y\}} D_{a, \alpha}
            &= \sum_{a \in \mathbb{F}_q} \sum_{c \in \mathbb{F}_q} (\otimes_{i=1}^m U^{c_i}) \otimes (\otimes_{i=1}^m U^{a_ic_i + \alpha_i}) + D_{y,\alpha} \\
            &= (\sum_{c \in \mathbb{F}_q, c \neq 0} (\otimes_{i=1}^m U^{c_i}) \otimes \sum_{a \in \mathbb{F}_q} (\otimes_{i=1}^m U^{a_ic_i+\alpha_i})) + q I_q \otimes (\otimes_{i=1}^m U^{\alpha_i}) + (\otimes_{i=1}^m U^{\alpha_i}) \otimes J_q \\
            &= (\otimes_{i=1}^m (J_q - I_q) \otimes J_q) + q I_q \otimes (\otimes_{i=1}^m U^{\alpha_i}) + (\otimes_{i=1}^m U^{\alpha_i}) \otimes J_q \\
            &= (\otimes_{i=1}^m (J_q - I_q + U^{\alpha_i}) \otimes J_q) + q I_q \otimes (\otimes_{i=1}^m U^{\alpha_i}),
        \end{align*}
        as required.

    \item
        The assertion follows by
        \begin{align*}
            \sum_{a,b\in \mathbb{F}_q\cup\{y\},a\neq b} P_{a}P_{b}&=\sum_{a,b\in \mathbb{F}_q\cup\{y\}} P_{a}P_{b}-\sum_{a\in \mathbb{F}_q\cup\{y\}} P_{a}^2\\
            &=(\sum_{a\in \mathbb{F}_q\cup\{y\}} P_{a})^2-\sum_{a\in \mathbb{F}_q\cup\{y\}} (2I_{2q+3}+V^{2\varphi(a)}+V^{-2\varphi(a)})\\
            &=(J_{2q+3}-I_{2q+3})^2-2(q+1)I_{2q+3}-(J_{2q+3}-I_{2q+3})\\
            &= 2q(J_{2q+3}-I_{2q+3}).
        \end{align*}
\end{enumerate}

\end{proof}

Now, we prove the main theorem of this section.
\begin{tm}\label{thm:N}
For any $\alpha \in \mathbb{F}_q$, $N_\alpha$ is an adjacency matrix of a Deza graph $D$ with parameters $(q^2(2q+3), 2q(q+1), 3q, 2q)$.
Furthermore, $N_0, N_1, \dots, N_{q-1}$ gives a decomposition of the complete graph $K_{q^2(2q+3)}$ into $(q-1)$ disjoint Deza graphs, each one of which is isomorphic to $D$.
\end{tm}
\begin{proof}
It is easy to see that $N_\alpha$ is a symmetric $(0,1)$-matrix for any $\alpha \in \mathbb{F}_q$
Applying Lemma~\ref{lem:CD}, we obtain
\begin{align*}
N_{\alpha}^2
&=\sum_{a,b \in \mathbb{F}_q\cup\{y\}} P_{a}P_{b} \otimes C_{a,\alpha} C_{b,\alpha}\\
&=\sum_{a \in \mathbb{F}_q \cup \{y\}} P_{a}^2 \otimes C_{a,\alpha}C_{a,\alpha} + \sum_{a,b \in \mathbb{F}_q\cup\{y\},a\neq b} P_{a}P_{b} \otimes C_{a,\alpha}C_{b,\alpha}\\
&=\sum_{a \in \mathbb{F}_q \cup \{y\}}(2 I_{2q+3}+V^{2\varphi(a)}+V^{-2\varphi(a)})\otimes q D_{a,0}+\sum_{a,b\in\mathbb{F}_q\cup\{y\},a\neq b} P_{a}P_{b}\otimes J_{q^2}\displaybreak[0]\\
&=2qI_{2q+3} \otimes (qI_{q^2} + J_{q^2}) + 2q(J_{2q+3}-I_{2q+3}) \otimes J_{q^2} + q \sum_{a \in \mathbb{F}_q \cup \{y\}} (V^{2\varphi(a)}+V^{-2\varphi(a)}) \otimes D_{a,0}\\
&=2q^2 I_{q^2(2q+3)} + 2q J_{q^2(2q+3)} + q \sum_{a \in \mathbb{F}_q \cup \{y\}} (V^{2\varphi(a)}+V^{-2\varphi(a)}) \otimes D_{a,0}\\
&=2q(q + 1) I_{q^2(2q+3)} + 3q M + 2q (J_{q^2(2q+3)} - I_{q^2(2q+3)} - M),
\end{align*}
Here $M = \sum_{a \in \mathbb{F}_q \cup \{y\}} (V^{2\varphi(a)}+V^{-2\varphi(a)}) \otimes D_{a,0}$, which is a symmetric $(0,1)$-matrix as one can easily see.
Hence $M$ is an adjacency matrix of a simple graph, and $N_{\alpha}$ is an adjacency matrix of a Deza graph with parameters $(q^2(2q+3), 2q(q+1),3q,2q)$.
This proves the first part of the theorem.

To prove the second part, we compute
\begin{align*}
    \sum_{\alpha \in \mathbb{F}_q} N_{\alpha}
    &= \sum_{\alpha \in \mathbb{F}_q} \sum_{a \in \mathbb{F}_q \cup \{y\}} P_a \otimes C_{a,\alpha} \\
    &= \sum_{a \in \mathbb{F}_q \cup \{y\}} P_a \otimes \sum_{\alpha \in \mathbb{F}_q} C_{a,\alpha} \\
    &= \sum_{a \in \mathbb{F}_q \cup \{y\}} P_a \otimes J_{q^2} \\
    &= (J_{2q+3} - I_{2q+3}) \otimes J_{q^2}.
\end{align*}
Hence, we have $\sum_{\alpha \in \mathbb{F}_q} N_{\alpha} + I_{2q+3} \otimes J_{q^2} = J_{q^2(2q+3)}$.
This completes the proof.
\end{proof}

\begin{prop}
    $N_0,N_1,\dots,N_{q-1}$ are commutative if and only if $q=2^m$.
\end{prop}
\begin{proof}
    Let $\alpha, \beta \in \mathbb{F}_q$ and $\alpha \neq \beta$.
    \begin{align*}
    &N_{\alpha} N_{\beta} \\
    &=\sum_{a \in \mathbb{F}_q \cup \{y\}} P_{a}^2 \otimes C_{a,\alpha}C_{a,\beta} + \sum_{a,b \in \mathbb{F}_q\cup\{y\},a\neq b} P_{a}P_{b} \otimes C_{a,\alpha}C_{b,\beta}\\
    &=\sum_{a \in \mathbb{F}_q \cup \{y\}}(2 I_{2q+3}+V^{2\varphi(a)}+V^{-2\varphi(a)})\otimes q D_{a,\alpha-\beta}+\sum_{a,b\in\mathbb{F}_q\cup\{y\},a\neq b} P_{a}P_{b}\otimes J_{q^2}\displaybreak[0]\\
    &=2qI_{2q+3} \otimes \sum_{a \in \mathbb{F}_q \cup \{y\}} D_{a,\alpha-\beta} + 2q(J_{2q+3}-I_{2q+3}) \otimes J_{q^2} + q \sum_{a \in \mathbb{F}_q \cup \{y\}} (V^{2\varphi(a)}+V^{-2\varphi(a)}) \otimes D_{a,\alpha-\beta}.
    \end{align*}
    If $p=2$, then $D_{a,\alpha-\beta}=D_{a,\beta-\alpha}$.
    It follows that $N_{\alpha} N_{\beta} = N_{\beta} N_{\alpha}$, as required.
    Now, we conclude the proof by showing that $N_{\alpha} N_{\beta} \neq N_{\beta} N_{\alpha}$ assuming $p \neq 2$.
    As a square block matrix of order $2q+3$, any diagonal entry of $N_{\alpha} N_{\beta}$ equals
    \begin{align*}
        2q \sum_{a \in \mathbb{F}_q \cup \{y\}} D_{a,\alpha-\beta}
        &= 2q( qI_q \otimes (\otimes_{i=1}^m U^{\alpha_i-\beta_i}) + (\otimes_{i=1}^m (J_q-I_q+U^{\alpha_i-\beta_i}) \otimes J_q).
    \end{align*}
    Again, as a square matrix of order $q$, any diagonal entry of this matrix equals $2q^2 (\otimes_{i=1}^m U^{\alpha_i-\beta_i})$.
    It is clear that $\otimes_{i=1}^m U^{\alpha_i-\beta_i} \neq \otimes_{i=1}^m U^{\beta_i - \alpha_i}$.
    Therefore, the diagonal entry of $N_{\alpha,\beta}$ (as a square block matrix of order $q(2q+3)$) is not equal to that of $N_{\beta, \alpha}$.
    This shows that $N_{\alpha, \beta} \neq N_{\beta, \alpha}$, as required.
\end{proof}


\begin{thebibliography}{40}

\bibitem{network}
R. Ahlswede, N. Cai, S.-Y. R. Li, R. W. Yeung, Network information flow, IEEE Trans. Inform. Theory 46 (2000), 1204--1216. 

\bibitem{AHS}
M. Araya, M. Harada, S. Suda, 
Quasi-unbiased Hadamard matrices and weakly unbiased Hadamard matrices: a coding-theoretic approach, Math. Comp. 304 (2017), 951--984. 

\bibitem{equitable}
F. Atik, On equitable partition of matrices and its applications, Linear Multilinear Algebra 68 (2020), 2143--2156. 

\bibitem{BI}
E. Bannai, T. Ito, Algebraic Combinatorics I: Association Schemes,
{Benjamin/Cummings, Menlo Park, CA,} 1984.

\bibitem{daniele}
D. Bartoli, A.-E. Riet, L. Storme, P. Vandendriessche, Improvement to the sunflower bound for a class of equidistant constant dimension subspace codes. J. Geom. 112 (2021), no. 1, paper no. 12, 9 pp.

%

\bibitem{bose}
R. C. Bose, Symmetric group divisible designs with the dual property, J. Stat. Plann. Inference 1 (1977), 87--101.





\bibitem{orb-Had}
D. Crnkovi\'c, R. Egan, A. \v Svob, Orbit matrices of Hadamard matrices and related codes, Discrete Math. 341 (2018), 1199--1209.

\bibitem{deza}
D. Crnkovi\'c, H. Kharaghani, S. Suda, A. \v Svob, New constructions of Deza digraphs, 
Int. J. Group Theory, 13 no. 3 (2024), 225--240.

\bibitem{dsa-as}
D. Crnkovi\'c, S. Rukavina, A. \v Svob, Self-orthogonal codes from equitable partitions of association schemes, J. Algebraic Combin. 55 (2022), 157--171.

\bibitem{dc-as-lcd}
D. Crnkovi\'c, A. \v Svob, LCD subspace codes, Des. Codes Cryptogr. 91 (2023), 3215--3226.

\bibitem{dc-as-gc}
D. Crnkovi\'c, A. \v Svob, Self-Orthogonal Codes from Deza Graphs, Normally Regular Digraphs and Deza Digraphs. Graphs and Combinatorics 40, 35 (2024). https://doi.org/10.1007/s00373-024-02763-y


\bibitem{network-book}
M.~Greferath, M.~O.~Pav\v cevi\'c, N.~Silberstein, M.~V\'azquez-Castro (eds.),
Network coding and subspace designs, Signals and Communication Technology, Springer, Cham, 2018.

\bibitem{ddgs}
W. H. Haemers, H. Kharaghani, M. A. Meulenberg, Divisible Design Graphs, J. Combin. Theory Ser. A 118 (2011), 978--992.


\bibitem{har_ton} 
M.~Harada, V.~D.~Tonchev, Self-orthogonal codes from symmetric designs with fixed-point-free automorphisms, 
Discrete Math. 264 (2003) 81--90.


\bibitem{heinlein-dcc}
D.~Heinlein, T.~Honold, M.~Kiermaier, S.~Kurz, A.~Wassermann, Classifying optimal binary subspace codes of length 8, constant dimension 4 and minimum distance 6, Des. Codes Cryptogr. 87 (2019) 375--391.

\bibitem{random}
T. Ho, M. M\'{e}dard, R. K\"{o}tter, D. R. Karger, M. Effros, J. Shi, B. Leong, 
A random linear network coding approach to multicast, IEEE Trans. Inform. Theory 52 (2006), 4413--4430. 
%

\bibitem{FEC} 
W.~C.~Huffman, V.~Pless, Fundamentals of Error-Correcting Codes, Cambridge University Press, 2003.

\bibitem{zj}
Z.~Janko, Coset enumeration in groups and constructions of symmetric designs, Combinatorics '90 (Gaeta, 1990), Ann. Discrete Math. 52 (1992), 275--277.


\bibitem{MUH-BT}
H. Kharaghani, S. Sasani, S. Suda,  Mutually unbiased Bush-type Hadamard matrices and association schemes, Electron. J. Combin. 22, No. 3 (2015), Research Paper P3.10, 11 p. 

\bibitem{KS}
H. Kharaghani, S. Suda,
Hoffman's coclique bound for normal regular digraphs, and nonsymmetric association schemes, 
{\sl Mathematics Across Contemporary Sciences}, 137--150, 
Springer Proc. Math. Stat., 190, Springer, Cham, 2017.  

\bibitem{KSlinked}
H. Kharaghani, S. Suda, 
Linked systems of symmetric group divisible designs,
J. Algebraic Combin. 47 (2017), no. 2, 319--343. 

\bibitem{KS2018}
H. Kharaghani, S. Suda,
Unbiased orthogonal designs, 
Des .Codes Cryptogr. 86, (2018), 1573–-1588. 


\bibitem{KSlinked2} 
H. Kharaghani, S. Suda, Linked system of symmetric group divisible designs of type II, 
Des .Codes Cryptogr. 87, (2019), 2341--2360. 


\bibitem{network-coding}
R.~K\"{o}tter, F.~Kschischang, Coding for errors and erasures in random network coding, IEEE Trans. Inform. Theory 54 (2008) 3579--3591.

\bibitem{leo-constant}
L. Hernandez Lucas, I. Landjev, L. Storme, P. Vandendriessche, A stability result and a spectrum result on constant dimension codes, Linear Algebra Appl. 621 (2021), 193--213.

\bibitem{massey}
J.~L.~Massey, Linear codes with complementary duals, Discrete Math. 106/107 (1992), 337--342.

\bibitem{mquwm}
H. Nozaki, S. Suda, Weighing matrices and spherical codes, J. Algebraic. Combin. 42 (2015), 283--291.


\bibitem{rank-metric}
D. Silva, F. R. Kschischang, R. K\"{o}tter,
A rank-metric approach to error control in random network coding, IEEE Trans. Inform. Theory 54 (2008), 3951--3967. 


\bibitem{zhang}
H. Zhang, X. Cao, Further constructions of cyclic subspace codes, Cryptogr. Commun. 13 (2021), 245-262.

\end{thebibliography}
\end{document}